\documentclass[12pt,reqno]{amsart}
\usepackage{latexsym,amssymb,amsmath,multicol, rotating,lscape}
\usepackage{letltxmacro}
\usepackage{tikz}
\usetikzlibrary{calc,decorations.markings}
\usepackage[]{hyperref}
\hypersetup{
	colorlinks=true,        
	linkcolor=blue,
	citecolor=red,
	urlcolor=teal,
	pdftitle={Title},
	pdfauthor={Anubhav Sharma and Lalit Vaishya},
}
\usepackage[noabbrev,nameinlink]{cleveref}

 2

\newtheorem{thm}{Theorem}[section]
\newtheorem{lem}[thm]{Lemma}
\newtheorem{prop}[thm]{Proposition}

\newcommand{\thmref}[1]{Theorem~\ref{#1}}
\newcommand{\lemref}[1]{Lemma~\ref{#1}}

\newcommand{\propref}[1]{Proposition~\ref{#1}}

\theoremstyle{remark}
\newtheorem{rmk}{Remark}[section]

\newcommand{\vect}[1]{\underline{#1}}
\numberwithin{equation}{section}

\begin{document}
	
	\title[Sparse Distribution of $\ell$-fold Coefficients]
	{Sparse Distribution of Coefficients of $\ell$-fold Product $L$-functions at Integers Represented by Quadratic Forms}

\author[Anubhav Sharma]{Anubhav Sharma}
\author[Mohit Tripathi]{Mohit Tripathi}
\author[Lalit Vaishya]{Lalit Vaishya}

\address[Anubhav Sharma, Lalit Vaishya]{Stat-Math Unit, Indian Statistical Institute, 7, S. J. S. Sansanwal Marg, New Delhi 110016, India.}

\email{\href{mailto:anubhav6595@gmail.com}{anubhav6595@gmail.com}}

\email{\href{mailto:lalitvaishya@gmail.com; lalitvaishya$\_$24v@isid.ac.in}{lalitvaishya@gmail.com; lalitvaishya$\_$24v@isid.ac.in}}

\address[Mohit Tripathi]{Department of Mathematics and Statistics, Texas Tech University, Lubbock, TX 79410-1042, USA.}
\email{\href{mailto: mohit.tripathi@ttu.edu}{ mohit.tripathi@ttu.edu}}


	\maketitle

\begin{abstract}
	Let $f \in S_{k}(\Gamma_{0}(N))$ be a normalized Hecke eigenform. We study the Fourier coefficients $\lambda_{f \otimes \cdots \otimes_{\ell} f}(n)$ of the $\ell$-fold product $L$-function for odd $\ell \ge 3$. Our focus is the distribution of this sequence over the sparse set of integers represented by a primitive, positive-definite binary quadratic form $Q$ of a fixed discriminant $D$. We establish an explicit upper bound for the summatory function of these coefficients, with dependencies on the weight, level, and discriminant. As a key application, we provide a bound for the first sign change of the sequence in this setting. We also generalize this result to find the first sign change among integers represented by any of the $h(D)$ forms of discriminant $D$, showing the bound improves as the class number increases.
\end{abstract}

\section{Introduction and statements of the results}

The study of Fourier coefficients of modular forms lies at the heart of modern number theory, connecting analysis, algebra, and geometry. Let $S_k(\Gamma_0(N))$ be the space of cusp forms of weight $k$ for the congruence subgroup $\Gamma_0(N)$. For a Hecke eigenform $f \in S_k(\Gamma_0(N))$, its Fourier series at infinity is given by $f(\tau) = \sum_{n=1}^{\infty} a_f(n) e^{2\pi i n \tau}$. The normalized Fourier coefficients, $\lambda_f(n) := a_f(n)/n^{(k-1)/2}$, are real-valued, multiplicative, and satisfy the celebrated Deligne \cite{Deligne:WeilI} bound $|\lambda_f(n)| \le d(n)$, where $d(n)$ is the divisor function.

Within the framework of the Langlands program, one associates a family of automorphic $L$-functions to $f$. A particularly important family is the $\ell$-fold product $L$-function, $L(s, f \otimes \cdots \otimes_{\ell} f)$, which is an $L$-function of degree $2^\ell$. For $f_1, \dots, f_\ell$, its Euler product is defined for $\Re(s) > 1$ by
$$L(s,f_{1}\otimes \cdots \otimes f_{\ell}) = \prod_{p}\prod_{\sigma}(1-\alpha_{f_{1}}^{\sigma(1)}(p)\cdots\alpha_{f_{\ell}}^{\sigma(\ell)}(p)p^{-s})^{-1},$$
where $\sigma$ runs over all maps from $\{1, \dots, \ell\}$ to $\{1, 2\}$, and $\alpha_f^{(1)}(p), \alpha_f^{(2)}(p)$ are the Satake parameters of $f$. When all forms are taken to be $f$, we denote the coefficients of the resulting Dirichlet series by $\lambda_{f \otimes \cdots \otimes_{\ell} f}(n)$. These coefficients are multiplicative, and at a prime $p$, they are given by the simple relation
\begin{equation}\label{eqn1}
\lambda_{f \otimes \cdots \otimes_{\ell} f}(p) = \lambda_f(p)^\ell.
\end{equation}
The analytic properties of these coefficients are deeply connected to those of symmetric power $L$-functions, as their Dirichlet series can be decomposed into a product of such functions.

A central and challenging problem in analytic number theory is to understand the distribution of arithmetic sequences over sparse sets of integers. While classical results describe behavior in arithmetic progressions, significant recent attention has been given to sequences such as the values of polynomials \cite{Blomer:2008} or quadratic forms. Of particular interest are integers represented by a primitive, positive-definite binary quadratic form $Q(x_1, x_2)$ of fixed discriminant $D < 0$. Studying the summatory function of Hecke eigenvalues $\lambda_f(n)$ where $n$ is restricted to values of $Q$ is an active area of research \cite{Acharya:2022}. The work of Vaishya \cite{Lalit-M}, for instance, explore moments of $\lambda_f(n)$ and coefficients of the triple product $L$-function in this setting. This paper extends this line of inquiry to the full family of odd $\ell$-fold product $L$-functions.

A particularly subtle aspect of these sequences is the distribution of their signs. For real coefficients, it is known that $\lambda_f(n)$ must change sign infinitely often. A quantitative version of this problem seeks to bound the first integer $n_f$ for which $\lambda_f(n) < 0$. This is a deep problem, considered a $GL_2$-analogue of finding the least quadratic non-residue. There has been significant progress in bounding $n_f$ in terms of the analytic conductor of $f$, with key results from Kohnen, Sengupta, Iwaniec, Kowalski, and Matomäki \cite{IwaniecKohnenSengupta:2007, KowalskiLauSoundWu:2010, Matomaki:2012}. This naturally leads to analogous questions for the coefficients $\lambda_{f \otimes \cdots \otimes_{\ell} f}(n)$, especially when the sequence is restricted to the sparse set of integers represented by a quadratic form $Q$.

In this article, we investigate the summatory function of the coefficients of odd $\ell$-fold product $L$-functions over the sparse set of square-free integers represented by a primitive binary quadratic form. Our first main result (\thmref{thm1.1}) provides an explicit upper bound for this sum, with dependencies on the weight $k$, level $N$, discriminant $D$, and the order $\ell$ of the product. As a key application, we leverage this estimate to establish an upper bound for the first sign change of this sequence in the same setting (\thmref{thm1.2}). Our results substantially generalize previous works, which were primarily focused on the case of $\lambda_f(n)$ itself or small values of $\ell$, and provide new insights into the intricate interplay between higher-rank automorphic forms and the arithmetic of quadratic fields.

Throughout this paper, we consider a primitive integral positive-definite binary quadratic form $Q(\vect{x})$, which is a reduced form of the type $Q(\vect{x}) = ax_1^2 + bx_1x_2 + cx_2^2$ for $\vect{x} = (x_1, x_2) \in \mathbb{Z}^2$. The coefficients $a,b,c$ are integers satisfying $\gcd(a,b,c)=1$, and the form has a fixed negative discriminant $D=b^2-4ac<0$. We specifically focus on discriminants $D$ for which the class number $h(D)$ is one. For a detailed background on these forms, we refer the reader to \cite[Chapter]{Cox}.

We consider the following sum for an upper-bound estimate:
\begin{equation} \label{eq:sum_S}
	S_{\ell}(f, Q; X ) := \sideset{}{^{\flat }}\sum^{}_{\substack{n=Q(\vect{x})\le X \\ \vect{x}\in\mathbb{Z}^2, \gcd(n,N)=1}} \lambda_{f \otimes \cdots f\otimes_{\ell} f}(n),
\end{equation}
where $\lambda_f(n)$ denotes the normalised $n$-th Fourier coefficient of a Hecke eigenform $f$ and the symbol $\flat$ means that the sum runs over all square-free integers. We also produce the explicit dependency in terms of weight $k$ and level $N$ of the Hecke eigenform $f$, and discriminant $D$ of $Q$. More precisely, we obtain the following estimate.

\begin{thm}\label{thm1.1}
	Let $\ell \ge $ be an odd positive integer and $f \in S_{k}(\Gamma_{0}(N))$ be a normalised Hecke eigenform and $Q$ be a reduced form of discriminant $D$ with the class number $h(D) = 1$. For sufficiently large $X > 0$ and any arbitrarily small $\epsilon > 0$, we have
	$$
	S_{\ell}(f,Q ; X ) = \sideset{}{^{\flat }}\sum^{}_{\substack{n=Q(\vect{x})\le X \\ \gcd(n,N)=1}} \lambda_{f \otimes \cdots f\otimes_{\ell} f}(n) \ll_{\epsilon} X^{1- \frac{1}{B}+ \epsilon} \left(N^{A} (k|D|^\frac{1}{2})^{B}\right)^{\frac{1}{B} + \epsilon}.
	$$
	where 
	\begin{equation}\label{ABF}
		A=\sum_{n=0}^{[\ell/2]} {\frac{(\ell-2n+1)(\ell-2n)}{\ell-n+1} {\ell \choose {n}}}  \quad \text{and} \quad  B= \sum_{n=0}^{[\ell/2]} {\frac{(\ell-2n+1)^{2}}{\ell-n+1} {\ell \choose {n}}}.
	\end{equation}
\end{thm}

\begin{rmk}
	If the class number $h(D) > 1$, the same estimate (obtained in \thmref{thm1.1}) holds for the sum
	\begin{equation} \label{eq:sum_SD}
		S_{\ell}(f, D; X) := \sideset{}{^{\flat }}\sum^{}_{\substack{n=Q(\vect{x})\le X \\ Q\in\mathcal{S}_D, \vect{x}\in\mathbb{Z}^2 \\ \gcd(n,N)=1}} \lambda_{f \otimes \cdots f\otimes_{\ell} f}(n),
	\end{equation}
	where $\mathcal{S}_D$ denotes the set of in-equivalent primitive integral positive-definite binary quadratic (reduced) forms of fixed discriminant with the class number $h(D) = \#\mathcal{S}_D$.
\end{rmk}

For an odd positive integer $\ell \ge 3$, let $n_{f \otimes \cdots \otimes_{\ell} f,I}$ denote the integer at which the first sign change of the eigenvalues occurs, i.e., the smallest integer $n \in \mathbb{N}$ with $\gcd(n, N) = 1$ such that $\lambda_{f \otimes \cdots \otimes_{\ell} f}(n) < 0$. Let $n_{f \otimes \cdots \otimes_{\ell} f,Q}$ be the least such integer that is also represented by $Q(\vect{x})$.
With these notations, we observe that $n_{f \otimes \cdots \otimes_{\ell} f,I} \le n_{f \otimes \cdots \otimes_{\ell} f,Q}$ as it is not necessary that $n_{f \otimes \cdots \otimes_{\ell} f,I}$ can always be represented by $Q(\vect{x})$ for some $\vect{x} \in \mathbb{Z}^2$. Hence, an upper bound estimate for $n_{f \otimes \cdots \otimes_{\ell} f,Q}$ will also work for $n_{f \otimes \cdots \otimes_{\ell} f,I}$. An estimate for $n_{f \otimes \cdots \otimes_{\ell} f,Q}$ is given by the following result.


\begin{thm}\label{thm1.2}
	Let $\ell \ge 3$ be an odd integer, $f \in S_{k}(\Gamma_{0}(N))$ a Hecke eigenform, and $Q$ a reduced form of discriminant $D$ with class number $h(D) = 1$. Then there exists a constant $u_{0, \ell} > 1$ such that
	$$
	n_{f \otimes \cdots f\otimes_{\ell} f, Q} \;\ll_{\epsilon}\; 
	\bigl(N^{A} k^{B}\bigr)^{\frac{1}{2u_{0, \ell}} + \epsilon} 
	\left(\frac{2 \pi}{w_D}\right)^{-\frac{2^{\ell-1} B}{u_{0, \ell}}} 
	|D|^{\,\frac{ (1 - 2^{\ell-1})B}{2 u_{0,\ell}}+\epsilon}.
	$$  
	Here $A$ and $B$ are as in \thmref{thm1.1}, and $w_D$ is as defined in \eqref{eq:8}. For each odd $\ell \ge 3$, the constants $u_{0,\ell}$ are explicitly computable; some values are listed in the following remark.
\end{thm}

\begin{rmk}
	For each odd $\ell \ge 3$, the values of $u_{0,\ell}$ can be explicitly computed. Some examples are:
	\[
	(A_3, B_3, u_{0,3}) = (5, 8, 2.235), \quad
	(A_5, B_5, u_{0,5}) = (22, 32, 5.268), \quad
	(A_7, B_7, u_{0,7}) = (93, 126, 1.25),
	\]
	and similarly for other odd $\ell$.
\end{rmk}

\begin{rmk}
	If the class number $h(D) > 1$ for the discriminant $D < 0$, then there are $h(D)$ many reduced forms corresponding to discriminants $D$. Moreover, let $n_{f \otimes \cdots f\otimes_{\ell} f,D}$ be the least integer among all $n \in \mathbb{N}$ such that $\lambda_{f \otimes \cdots f\otimes_{\ell} f}(n) < 0$, with $\gcd(n, N) = 1$ and $n$ is represented by some of them from $h(D)$ many reduced forms corresponding to discriminants $D$. Then,
	$$
	n_{f \otimes \cdots f\otimes_{\ell} f,D} \;\ll_{\epsilon}\; 
	\bigl(N^{A} k^{B}\bigr)^{\frac{1}{2u_{0, \ell}} + \epsilon} 
	\left(\frac{2 \pi}{w_D}\right)^{-\frac{2^{\ell-1} B}{u_{0, \ell}}} 
	h(D)^{-\frac{(2^{\ell-1}+1) B}{2 u_{0,\ell}}} 
	|D|^{-\frac{(2^{\ell-1}-1) B}{2 u_{0,\ell}}+\epsilon}.
	$$
	\end{rmk}

The paper is organized as follows. In the next section,  we provide some key ingredients which are necessary to prove our results. Finally, in the last section, we prove our results.

\section{Preparatory Results}
To obtain our main result, we first analyze the sum
\begin{equation*}
	S_{\ell}(f, Q; X ) := \sideset{}{^{\flat }}\sum_{\substack{n=Q(\vect{x})\le X \\ \vect{x}\in\mathbb{Z}^2, \gcd(n,N)=1}} \lambda_{f \otimes \cdots f\otimes_{\ell} f}(n),
\end{equation*}
where the symbol $\flat$ indicates that the sum runs over square-free integers $n$, and $\lambda_{f \otimes \cdots f\otimes_{\ell} f}(n)$ are the Hecke eigenvalues of the $\ell$-fold tensor product of a Hecke eigenform $f$. Let $r_Q(n)$ denote the number of representations of an integer $n$ by a binary quadratic form $Q$. Since $r_Q(n) \ge 0$, we can rewrite the sum over values of $n$ as
\begin{equation} \label{eq:7}
	S_{\ell}(f,Q; X) = \sideset{}{^{\flat }}\sum_{\substack{n\le X \\ \gcd(n,N)=1}} \lambda_{f \otimes \cdots f\otimes_{\ell} f}(n) r_Q(n).
\end{equation}
The generating function for $r_Q(n)$ is the theta series associated to $Q$,
$$	\theta_Q(\tau) := \sum_{\vect{x}\in\mathbb{Z}^2} q^{Q(\vect{x})} = \sum_{n=0}^{\infty} r_Q(n)q^n, \quad q := e^{2\pi i\tau},$$
which is a modular form of weight 1, specifically $\theta_Q(\tau) \in M_1(\Gamma_0(|D|), \chi_D)$, where $\chi_D(d) := \left(\frac{D}{d}\right)$ is the Kronecker symbol associated to the discriminant $D$ of $Q$ (see \cite[Theorem 10.9]{Iwaniec}). It is well-known that $r_Q(n) \ll_{\epsilon} n^{\epsilon}$ for any $\epsilon > 0$.

We consider positive definite quadratic forms $Q(\vect{x})$ with negative discriminant $D$ and class number $h(D)=1$. For such forms, there is a simple formula for the representation numbers \cite[Section 11.2]{Iwaniec}:
\begin{equation} \label{eq:8}
	r_Q(n) = w_D \sum_{d|n} \chi_D(d), \quad \text{where } w_D =
	\begin{cases}
		6 & \text{if } D = -3, \\
		4 & \text{if } D = -4, \\
		2 & \text{if } D < -4,
	\end{cases}
\end{equation}
and $w_D$ is the number of units in the quadratic field $\mathbb{Q}(\sqrt{D})$. This formula depends only on the discriminant $D$, not on the specific coefficients of the reduced form $Q$ \cite[Section 2]{Lalitindag}. We define the primitive representation function as $r^*_Q(n) := \sum_{d|n} \chi_D(d)$, so that $r_Q(n) = w_D r^*_Q(n)$. Substituting this into \eqref{eq:7} gives
$$	S_{\ell}(f,Q; X) = w_D \sideset{}{^{\flat }}\sum_{\substack{n\le X \\ \gcd(n,N)=1}} \lambda_{f \otimes \cdots f\otimes_{\ell} f}(n)r^*_Q(n).$$
To estimate this sum using analytic methods, we introduce the associated Dirichlet series
\begin{equation} \label{eq:9}
	L_{\ell}(f,Q; s) := \sideset{}{^{\flat }}\sum_{\substack{n\ge 1 \\ \gcd(n,N)=1}} \frac{\lambda_{f \otimes \cdots f\otimes_{\ell} f}(n)r^*_Q(n)}{n^s},
\end{equation}
which converges absolutely and uniformly for $\Re(s) > 1$. Our strategy is to decompose $L_{\ell}(f,Q; s)$ in terms of known Hecke $L$-functions. We now recall the necessary definitions.

Let $f(\tau) = \sum_{n=1}^{\infty} \lambda_f(n)n^{\frac{k-1}{2}} q^n \in S_k(\Gamma_0(N))$ be a normalized Hecke eigenform. The associated Hecke $L$-function is given by
\begin{equation} \label{eq:10}
	L(s, f) := \sum_{n\ge 1} \frac{\lambda_f(n)}{n^s} = \prod_{p|N} \left( 1 - \frac{\lambda_f(p)}{p^s} \right)^{-1} \prod_{p\nmid N} \left( 1 - \frac{\lambda_f(p)}{p^s} + \frac{1}{p^{2s}} \right)^{-1},
\end{equation}
for $\Re(s) > 1$. The completed $L$-function is
\begin{equation} \label{eq:11}
	\Lambda(s, f) := \left(\frac{\sqrt{N}}{2\pi}\right)^s \Gamma(s + k - 1) L(s, f),
\end{equation}
which admits an analytic continuation to the entire complex plane and satisfies a functional equation.

For a Dirichlet character $\chi$ of modulus $m$, the twist of $f$ by $\chi$ is the form $f \otimes \chi(\tau) := \sum_{n=1}^{\infty} \lambda_f(n)\chi(n) n^{\frac{k-1}{2}} q^n \in S_k(\Gamma_0(M), \chi^2)$ for some level $M$ dividing $Nm^2$. The associated twisted Hecke $L$-function is
\begin{equation} \label{eq:12}
	L(s, f \otimes \chi) := \sum_{n\ge 1} \frac{\lambda_f(n)\chi(n)}{n^s} = \prod_{p|M} \left( 1 - \frac{\lambda_f(p)\chi(p)}{p^s} \right)^{-1} \prod_{p\nmid M} \left( 1 - \frac{\lambda_f(p)\chi(p)}{p^s} + \frac{1}{p^{2s}} \right)^{-1},
\end{equation}
which converges absolutely for $\Re(s) > 1$. Its completed version,
\begin{equation} \label{eq:13}
	\Lambda(s, f \otimes \chi) := \left(\frac{\sqrt{M}}{2\pi}\right)^s \Gamma(s + k - 1)L(s, f \otimes \chi),
\end{equation}
also has an analytic continuation to $\mathbb{C}$ and satisfies a functional equation \cite[Section 7.2]{Iwaniec}.

We also require the properties of symmetric power $L$-functions. For $m \ge 2$, the $m^{th}$-symmetric power $L$-function associated to $f$ is given by
\begin{equation}\label{Symf}
	L(s,\operatorname{sym}^{m}f) := \prod_{p} \prod_{j=0}^{m} \left(1-{\alpha_p}^{m-j}{\beta_p}^{j} {p^{-s}}\right)^{-1} = \zeta(ms)\sum_{n=1}^\infty \frac{\lambda_{f}(n^{m})}{n^s} = \sum_{n=1}^\infty \frac{\lambda_{\operatorname{sym}^{m}f}(n)}{n^s},
\end{equation}
where $\lambda_{\operatorname{sym}^{m}f}(n)$ is a multiplicative function. At prime values, it is given by
\begin{equation}\label{SymLf}
	\lambda_{\operatorname{sym}^{m}f}(p) = \lambda_{f}(p^{m}) = \sum_{j=0}^{m} {\alpha_p}^{m-j}{\beta_p}^{j}.
\end{equation}
Here, for each prime $p \nmid N$, $\alpha_p$ and $\beta_p$ are the Satake parameters of $f$, satisfying $\alpha_p + \beta_p = \lambda_f(p)$ and $\alpha_p \beta_p = 1$. The Archimedean factor of $L(s,\operatorname{sym}^{m}f)$ is defined as
\begin{equation}
	L_{\infty}(s, \operatorname{sym}^{m}f) := \begin{cases}
		\displaystyle{\prod_{v=0}^{p}} \Gamma_{{\mathbb C}}\left(s +\left(v+\frac{1}{2}\right)(k-1)\right) & \text{if } m = 2p +1, \\
		\Gamma_{\mathbb R}(s +\delta_{2 \nmid p}) \displaystyle{\prod_{v=1}^{p}} \Gamma_{{\mathbb C}}(s +v(k-1)) & \text{if } m = 2p,
	\end{cases}
\end{equation}
where $\Gamma_{\mathbb R}(s) = {\pi}^{-s/2} \Gamma(s/2)$, $\Gamma_{\mathbb C}(s) = 2 ({2\pi})^{-s} \Gamma(s)$, and $\delta_{2 \nmid p}=1$ if $2 \nmid p$ and $0$ otherwise. We define the completed $L$-function as
$$\Lambda(s, \operatorname{sym}^{m}f) := N^{ms/2}L_{\infty}(s, \operatorname{sym}^{m}f) L(s, \operatorname{sym}^{m}f).$$
$\Lambda(s, \operatorname{sym}^{m}f)$ is an entire function on $\mathbb{C}$ and satisfies the functional equation
$$
\Lambda(s, \operatorname{sym}^{m}f) = \epsilon_{\operatorname{sym}^{m}f} \Lambda(1-s, \operatorname{sym}^{m}f),
$$
where $\epsilon_{\operatorname{sym}^{m}f} = \pm1$. For details, we refer to \cite[Section 3.2.1]{Cog-Mic}. From Deligne's bound, the coefficients satisfy
$$
|\lambda_{\operatorname{sym}^{m}f}(n)| \le d_{m+1}(n) \ll_{\epsilon} n^{\epsilon}
$$
for any $\epsilon >0$, where $d_{m}(n)$ denotes the number of ways to write $n$ as a product of $m$ positive integers. We make the convention that $L(s, \operatorname{sym}^{1}f) = L(s, f)$.

\begin{rmk}
	For a classical holomorphic Hecke eigenform $f$, Cogdell and Michel \cite{Cog-Mic} have given an explicit description of the analytic continuation and functional equation for $L(s,\operatorname{sym}^{m}f)$. More recently, the work of Newton and Thorne \cite{NewtonThorne:I, NewtonThorne:II} has established the long-conjectured automorphy of all symmetric powers $\operatorname{sym}^{m}f$, showing they correspond to cuspidal automorphic representations of $GL_{m+1}(\mathbb{A}_{\mathbb{Q}})$. This provides a complete understanding of their analytic properties.
\end{rmk}

\begin{lem}[{\cite[Lemma 3.3]{Lalit-M}}]\label{ChebpolyLem}
	Let $\ell \in \mathbb{N}$ and $j$ with $0 \le j \le \ell$. Define
	$$
	A_{\ell,j} :=
	\begin{cases}
		\binom{\ell}{\tfrac{\ell-j}{2}} - \binom{\ell}{\tfrac{\ell-j}{2}-1}, & \text{if } j \equiv \ell \pmod{2}, \\[6pt]
		0, & \text{otherwise}.
	\end{cases}
	$$
	Then
	$$
	x^{\ell} = \sum_{j=0}^{\ell} A_{\ell,j} \, T_{\ell-j}(x),
	$$
	where $T_{m}(2x) := U_{m}(x)$ and $U_{m}(x)$ is the $m^{\text{th}}$ Chebyshev polynomial of the second kind.
\end{lem}

\subsection*{Decomposition} The relationship between $L_{\ell}(f,Q; s)$ and these various Hecke $L$-functions is established in the following lemma.
\begin{lem}\label{LDecomp}
	Let  $ \ell \in \mathbb{N}.$ we have the following decomposition:
	\begin{equation*} \label{DecompositionL}
		\begin{split}
			L_{\ell}(f,Q; s) & = L_{\ell}(s) \times U_{\ell}(s), \\ 
		\end{split}
	\end{equation*}
	where for each odd $\ell$,
	\begin{equation}\label{LOdd}
		\begin{split}
			L_{\ell}(s) & =   \prod_{n=0}^{[\ell/2]} \left({L(s, \operatorname{sym}^{\ell-2n}f)}^{\left({\ell \choose n}- {\ell \choose {n-1}}\right)}  L(s, \operatorname{sym}^{\ell-2n}f \times \chi_{D})^{\left({\ell \choose n}- {\ell \choose {n-1}}\right)} \right),       \\
		\end{split}
	\end{equation}
	and for each even  $\ell$,
	\begin{equation}\label{LEven}
		\begin{split}
			L_{\ell}(s) & = \zeta(s)^{\left({\ell \choose \ell/2}- {\ell \choose {\ell/2-1}}\right)} L(s, \chi_{D})^{\left({\ell \choose \ell/2}- {\ell \choose {\ell/2-1}}\right)} \\
			& \qquad \times  \prod_{n=0}^{[\ell/2]-1} \left({L(s, \operatorname{sym}^{\ell-2n}f)}^{\left({\ell \choose n}- {\ell \choose {n-1}}\right)}  L(s, \operatorname{sym}^{\ell-2n}f \times \chi_{D})^{\left({\ell \choose n}- {\ell \choose {n-1}}\right)} \right),       \\
		\end{split}
	\end{equation}
	and $U_{\ell}(s)$  is given in  terms of  an Euler product which converges absolutely and uniformly for $\Re(s)>\frac{1}{2}$,  and $U_{\ell}(s) \neq 0$ for $\Re(s)=1$.
\end{lem}
\begin{proof}
	From \eqref{eqn1}, we see that, for any prime $p$, 
	$$
	\lambda_{f \otimes f \otimes \cdots \otimes_{\ell} f}(p) \;=\; \lambda_{f}^{\ell}(p).
	$$
	By Deligne's estimate, $\lambda_{f}(p) = 2\cos \theta$ and 
	$$
	\lambda_{f}(p^{\ell}) = T_{\ell}(2 \cos \theta),
	$$
	where $T_{\ell}(2x)= U_{\ell}(x)$ and $U_{\ell}(x)$ is the $\ell^{\text{th}}$ Chebyshev polynomial of the second kind.  
	From \lemref{ChebpolyLem}, we get
	\begin{equation}\label{FCRel}
		\lambda_{f \otimes f \otimes \cdots \otimes_{\ell} f}(p) 
		= \lambda_{f}^{\ell}(p)
		= \sum_{n=0}^{\lfloor \ell/2 \rfloor} 
		\Biggl( \binom{\ell}{n} - \binom{\ell}{n-1} \Biggr) 
		\lambda_{\operatorname{sym}^{\ell-2n} f}(p).
	\end{equation}
	We know that $\lambda_{f}(n)$ and 
	$$
	r^{*}_{Q}(n) := \sum_{d \mid n} \chi_{D}(d)
	$$
	are multiplicative functions. Thus, $L_{\ell}(f,Q; s)$ has an Euler product expansion
	$$
	L_{\ell}(f,Q; s) 
	= w_{D} \sideset{}{^{\flat}} \sum_{n \ge 1} \frac{\lambda_{f}(n)^{\ell}\, r^{*}_{Q}(n)}{n^{s}}
	= w_{D} \prod_{p} \left(1+ \frac{\lambda_{f}(p)^{\ell} \, r^{*}_{Q}(p)}{p^s} \right),
	\qquad \Re(s)>1.
	$$
	From \eqref{FCRel}, we have
	\begin{align*}
		\lambda_{f \otimes f \otimes \cdots \otimes_{\ell} f}(p)\, r^{*}_{Q}(p)
		&= \lambda_{f}^{\ell}(p)\, r^{*}_{Q}(p) \\
		&= \Biggl( \sum_{n=0}^{\lfloor \ell/2 \rfloor} 
		\Bigl( \binom{\ell}{n} - \binom{\ell}{n-1} \Bigr) 
		\lambda_{\operatorname{sym}^{\ell-2n} f}(p) \Biggr)(1+\chi_{D}(p)) \\
		&= \sum_{n=0}^{\lfloor \ell/2 \rfloor} 
		\Bigl( \binom{\ell}{n} - \binom{\ell}{n-1} \Bigr) 
		\lambda_{\operatorname{sym}^{\ell-2n} f}(p) \\
		&\quad + \sum_{n=0}^{\lfloor \ell/2 \rfloor} 
		\Bigl( \binom{\ell}{n} - \binom{\ell}{n-1} \Bigr) 
		\lambda_{\operatorname{sym}^{\ell-2n} f}(p)\,\chi_{D}(p).
	\end{align*}
	For $\Re(s)>1$, define
	$$
	L_{\ell}(s) :=
	\prod_{n=0}^{\lfloor \ell/2 \rfloor}
	\Bigl(L(s,\operatorname{sym}^{\ell-2n}f)^{\binom{\ell}{n}-\binom{\ell}{n-1}}
	\;L(s,\operatorname{sym}^{\ell-2n}f \times \chi_{D})^{\binom{\ell}{n}-\binom{\ell}{n-1}} \Bigr).
	$$
	This admits an Euler product of the form
	$$
	\prod_{p} \left(1+ \frac{A(p)}{p^s} + \frac{A(p^{2})}{p^{2s}} + \cdots \right), 
	\qquad A(p) := -\lambda_{f}(p)^{\ell}\, r^{*}_{D}(p).
	$$
	Moreover, define a sequence $\{B(p^{r})\}$ for $r \ge 1$ by
	$$
	B(p) = 0, 
	\qquad B(p^{r}) = A(p^{r}) + A(p^{r-1}) \lambda_{f}(p)^{\ell} r^{*}_{D}(p) 
	\quad (r \ge 2).
	$$
	It is easy to see that $B(n) \ll n^{\epsilon}$ for any $\epsilon>0$.  
	The associated Euler product for $U_{\ell}(s)$ is
	$$
	U_{\ell}(s) = w_{D} \prod_{p} \left(1+ \frac{B(p)}{p^s} + \frac{B(p^{2})}{p^{2s}} + \cdots \right),
	$$
	with $B(p^{2}) = A(p^{2}) - \lambda_{f}(p)^{2\ell} (r^{*}_{D}(p))^{2}$.  
	Therefore,
	$$
	L_{\ell}(f,Q; s) = L_{\ell}(s)\times U_{\ell}(s).
	$$
	This completes the proof.
\end{proof}

\subsection*{Convexity Bounds and Mean Square Estimates}
In this subsection, we recall certain analytic properties of $L$-functions,
which will be employed to derive an upper bound for the sum
$
S_{\ell}(f,Q;X)
$.
\begin{lem}\cite[Chapter 5]{HIwaniec}
	Let $f \in S_{k}(\Gamma_{0}(N))$ be a primitive cusp form and $\epsilon > 0$ be an arbitrarily small real number. The hybrid
	convexity bound of Hecke L-function and twisted Hecke L-function is given by:
	\begin{equation}
		\label{eq:15}
		L(s, f) \ll_{\epsilon} \left( Nk^{2}(|s| + 3)^{2}\right)^{\frac{1}{4}+\epsilon}
		\quad \text{and} \quad
		L(s, f \otimes \chi_{D}) \ll_{\epsilon} \left( Nk^{2}|D|^{2}(|s| + 3)^{2}\right)^{\frac{1}{4}+\epsilon}
	\end{equation}
	on the line $\Re(s) = \frac{1}{2} + \epsilon$. The second integral moment is given by
	\begin{gather*}
		\label{eq:16}
		\int_{T}^{2T} \left| L\left(\frac{1}{2} + \epsilon + it, f\right) \right|^{2} dt \ll_{\epsilon} (Nk^{2}T^{2})^{\frac{1}{2}+\epsilon} \\
		\quad \text{and} \quad
		\int_{T}^{2T} \left| L\left(\frac{1}{2} + \epsilon + it, f \otimes \chi_{D}\right) \right|^{2} dt \ll_{\epsilon} (Nk^{2}|D|^{2}T^{2})^{\frac{1}{2}+\epsilon}
	\end{gather*}
	uniformly for any $|T| \geq 1$.
\end{lem}

\begin{lem}\cite[Chapter 5]{HIwaniec}
	Let $f \in S_{k}(\Gamma_{0}(N))$ be a primitive cusp form and $\operatorname{sym}^{m}f$ denote the $m$-th symmetric power lift of $f$. Then, for any arbitrarily small $\epsilon >0$, we have
	\begin{equation*}
		L(\sigma+it, \operatorname{sym}^{m}f) \ll_{\epsilon}
		\begin{cases}
			(N^{m} k^{m}(1+|t|)^{m+1})^{\max \left\{\frac{1}{2}(1-\sigma), 0\right\} +\epsilon} & m \text{ is even} \\
			(N^{m} k^{m+1}(1+|t|)^{m+1})^{\max \left\{\frac{1}{2}(1-\sigma), 0\right\} +\epsilon} & m \text{ is odd, } m \ge 3
		\end{cases}
	\end{equation*}
	and for the twisted $L$-function:
	\begin{equation*}
		L(\sigma+it, \operatorname{sym}^{m}f\otimes \chi_{D}) \ll_{\epsilon}
		\begin{cases}
			(N^{m}|D|^{m+1} k^{m}(1+|t|)^{m+1})^{\max \left\{\frac{1}{2}(1-\sigma), 0\right\} +\epsilon} & m \text{ is even} \\
			(N^{m} |D|^{m+1} k^{m+1}(1+|t|)^{m+1})^{\max \left\{\frac{1}{2}(1-\sigma), 0\right\} +\epsilon} & m \text{ is odd, } m \ge 3
		\end{cases}.
	\end{equation*}
	Furthermore, let $F=\operatorname{sym}^{m}f $ and $\mathcal{Q}_F$ be the conductor of $F$. The second integral moment is given by
	\begin{equation*}
		\int_{T}^{2T} \left|L \left(\sigma+it, F\right)\right|^{2} dt \ll_{\epsilon} (\mathcal{Q}_F(1+|t|)^{m+1})^{(1-\sigma)+\epsilon}
	\end{equation*}
	uniformly for $\frac{1}{2} \le \sigma \le 1$ and $|t| \ge 1$. 
\end{lem}

\subsection*{Mean value of a multiplicative function over the integers represented by a reduced form of discriminant $D$ with $h(D) = 1$}

Let $\eta$ and $N$ be two fixed positive integers. We define the sum
\[
E_{\eta}(X) \;=\; \sideset{}{^{\flat}}\sum_{\substack{n = Q(\vect{x}) \leq X \\ \vect{x} \in \mathbb{Z}^2 \\ \gcd(n,N) = 1}} \eta^{\omega(n)},
\]
where $\omega(n)$ denotes the number of distinct prime divisors of a positive integer $n$.  
This sum can be expressed in terms of $r^*_Q(n)$ as follows:
\begin{equation} \label{eq:17}
	E_{\eta}(X) 
	= \sideset{}{^{\flat}}\sum_{\substack{n = Q(\vect{x}) \leq X \\ \vect{x} \in \mathbb{Z}^2 \\ \gcd(n,N) = 1}} \eta^{\omega(n)}
	= \sum_{\substack{n \leq X \\ \gcd(n,N) = 1}} \mu^2(n)\,\eta^{\omega(n)}\, r^*_Q(n).
\end{equation}
%
Applying Perron’s formula to the decomposition \eqref{eq:17} and performing a standard contour-shifting argument yields an estimate for $E_{\eta}(X)$. The precise result is stated in the following proposition.

\begin{prop}[{\cite[Proposition 2.3]{Lalitindag}}]\label{prop:2.3}
	Let $E_{\eta}(X)$ be as in \eqref{eq:17}. There exists an absolute constant $C = C(\eta)$ such that
	\begin{equation} \label{eq:20}
		E_{\eta}(X) = \frac{P(1)\, L(1, \chi_D)^{\eta}}{\Gamma(\eta)}\, X (\log X)^{\eta-1} 
		\left( 1 + O_{\eta}\Big( \frac{L_N^{2e\eta+2}\sqrt{N}}{\log X} \Big) \right),
	\end{equation}
	uniformly for $N \ge 1$ and $X \ge \exp(C L_N^{2e\eta+2})$, where $L_N = \log(\omega(N) + 3)$, and
	\begin{equation}\label{eq:19}
		P(s) 
		= \prod_{p \mid N} \big(1 - p^{-s}\big)^{\eta} \big(1 - \chi_D(p)\, p^{-s}\big)^{\eta}
		\prod_{p \nmid N} \big(1 - p^{-s}\big)^{\eta} \big(1 - \chi_D(p)\, p^{-s}\big)^{\eta} 
		\Big( 1 + \frac{\eta(1 + \chi_D(p))}{p^s} \Big).
	\end{equation}
\end{prop}

\begin{rmk}
	Proposition~\ref{prop:2.3} remains valid for $h(D) > 1$.
\end{rmk}

Let $Y$ be a positive real number. By Deligne's estimate, for each prime $p$ we have
$$
\lambda_{f}(p) = 2 \cos \theta_{p}, 
$$
and more generally, for any positive integer $m$,
$$
\lambda_{f}(p^{m}) = \frac{\sin((m+1)\theta_{p})}{\sin \theta_{p}}.
$$
We define a multiplicative function $h_{Y}$ supported on square-free positive integers as follows:
\begin{equation}\label{Auxfun}
	h_{Y}(p) =
	\begin{cases}
		\alpha\Bigl( \dfrac{\log p}{\log Y} \Bigr), & \text{if } p \le Y \text{ and } p \nmid N, \\[2mm]
		-2, & \text{if } p > Y \text{ and } p \nmid N, \\[1mm]
		0, & \text{if } p \mid N,
	\end{cases}
\end{equation}
where $\alpha : [0, 1] \to [-2,2]$ is a step function defined by
$$
\alpha_0 := \alpha(0) = 2, \qquad
\alpha(t) = 2 \cos\Bigl( \frac{\pi}{m+1} \Bigr) \text{ for } \frac{1}{m+1} < t \le \frac{1}{m}, \quad m \in \mathbb{N}.
$$
We shall make use of the function $h_{Y}$ to obtain a lower bound for the sum
$
S_{\ell}(f,Q; Y^u)
$
for some $u \ge 1$, following the approach in \cite{asif}.

\begin{prop}\label{PropLower}
	Let $U \ge 1$ be a real number,   $h_{Y}(n)$ (defined in \eqref{Auxfun}) and $\alpha(t)$ be as above with $\alpha_{0}>0$. Let $N \le X^{U}$ be a positive integer. Then, we have
	\begin{equation*}
		\begin{split}
			\sum_{n \le Y^{u} \atop \gcd(n,N) =1 }   h_{Y}^{\ell}(n)r^*_Q(n)  &= (\sigma(u) + o_{\alpha, \ell, U}(1)) \frac{P(1)L(1,\chi_{D})^{\alpha_{0}^{\ell}}}{\Gamma(\alpha_{0}^{\ell})} (\log Y^{u})^{\alpha_{0}^{\ell}-1} Y^{u} 
		\end{split}
	\end{equation*}
	uniformly for $u \in \left[\frac{1}{U}, U \right]$ where $P(s)$ is given in \eqref{eq:19}, and 
	\begin{equation*}
		\begin{split}
			\sigma(u)  &= {u}^{\alpha_{0}^{\ell}-1} + \sum_{j=1}^{\infty} \frac{(-1)^{j}}{j}  I_{j}(u)  \\
		\end{split}
	\end{equation*}
	with
	\begin{equation*}
		\begin{split}
			I_{j}(u) &= \int_{\Delta_{j}} (u- t_{1}- t_{2}- \cdots- t_{j})^{\alpha_{0}^{\ell}-1} \prod_{i=1}^{j} (\alpha_{0}^{\ell}- \alpha^{\ell}(t))  \frac{dt_{1} dt_{2} \cdots dt_{j}}{ {t_{1}t_{2} \cdots t_{j}}} \\
		\end{split}
	\end{equation*}
	and
	\begin{equation*}
		\begin{split}
			& \Delta_{j} = \{ ( t_{1}, t_{2}, \cdots, t_{j}) \in [0, \infty) \mid  |t_{1}+ t_{2}+ \cdots+t_{j}| \le u \}.
		\end{split}
	\end{equation*}
\end{prop}

\begin{proof}
	We follow the argument of \cite[Lemma~6]{Matomaki:2012} to give a brief proof of Proposition~\ref{PropLower}.  
	Starting from an inclusion–exclusion type identity, we have
\begin{align*}
		\sum_{\substack{n \leq Y^u \\ \gcd(n,N)=1}} h_Y^{\ell}(n)\, r^*_Q(n)
		&=
		\sum_{\substack{n \leq Y^u \\ \gcd(n,N)=1}} (\alpha_0^{\ell})^{\omega(n)}\, r^*_Q(n) \\
		&\quad + \sum_{j=1}^{\infty} \frac{(-1)^j}{j} 
		\sum_{\substack{p_1 \cdots p_j \leq Y^u \\ \gcd(p_1 \cdots p_j , N)=1}}
		\Biggl[ \prod_{i=1}^{j} \bigl(\alpha_0^{\ell} - h_Y^{\ell}(p_i)\bigr) \, r^*_Q(p_i) \Biggr] \\
		&\qquad \times 
		\sum_{\substack{n \leq Y^u/(p_1 \cdots p_j) \\ \gcd(n,N)=1}} (\alpha_0^{\ell})^{\omega(n)}\, r^*_Q(n).
\end{align*}
Applying \propref{prop:2.3}, we obtain
\begin{align*}
		\sum_{\substack{n \leq Y^u \\ \gcd(n,N)=1}} h_Y^{\ell}(n)\, r^*_Q(n)
		&=
		\frac{P(1)\, L(1,\chi_D)^{\alpha_0^{\ell}}}{\Gamma(\alpha_0^{\ell})}\, Y^u
		\Biggl[
		(\log Y^u)^{\alpha_0^{\ell}-1} \\
		&\quad + \sum_{j=1}^{\infty} \frac{(-1)^j}{j}
		\sum_{\substack{p_1 \cdots p_j \leq Y^u \\ \gcd(p_1 \cdots p_j , N)=1}}
		\prod_{i=1}^{j} \bigl(\alpha_0^{\ell} - h_Y^{\ell}(p_i)\bigr)\, r^*_Q(p_i)
		\frac{(\log \frac{Y^u}{p_1 \cdots p_j})^{\alpha_0^{\ell}-1}}{p_1 \cdots p_j} \\
		&\quad + o_{U,\alpha}\!\big((\log Y^u)^{\alpha_0^{\ell}-1}\big)
		\Biggr].
	\end{align*}
	Since the class number $h(D) = 1$, the Chebotarev density theorem implies that the density of primes
	represented by $Q$ of discriminant $D$ is $1/2$, and for such primes we have
	$r^*_Q(p) = 2$ and $0$ otherwise (see \cite[Eq.~1.1 \& 1.2]{asif}). Therefore,
	\begin{align*}
		\sum_{\substack{n \leq Y^u \\ \gcd(n,N)=1}} h_Y^{\ell}(n)\, r^*_Q(n)
		&=
		\frac{P(1)\, L(1,\chi_D)^{\alpha_0^{\ell}}}{\Gamma(\alpha_0^{\ell})}\, Y^u
		\Biggl[
		(\log Y^u)^{\alpha_0^{\ell}-1} \\
		&\quad + \sum_{j=1}^{\infty} \frac{(-1)^j}{j}
		\sum_{\substack{p_1 \cdots p_j \leq Y^u \\ \gcd(p_1 \cdots p_j , N)=1}}
		\prod_{i=1}^{j} \bigl(\alpha_0^{\ell} - h_Y^{\ell}(p_i)\bigr)
		\frac{(\log \frac{Y^u}{p_1 \cdots p_j})^{\alpha_0^{\ell}-1}}{p_1 \cdots p_j} \\
		&\quad + o_{U,\alpha}\!\big((\log Y^u)^{\alpha_0^{\ell}-1}\big)
		\Biggr].
	\end{align*}
	This expression matches precisely the one obtained in the proof of \cite[Lemma~6]{Matomaki:2012}.  
	Finally, removing the coprimality condition and applying the prime number theorem yields the desired result.
\end{proof}
\begin{rmk}
	Let $\alpha$ be a step function defined as
	$$
	\alpha: [0, \infty) \longrightarrow \mathbb{R}, \quad 
	\alpha(t) = \alpha_{k} \ \text{for} \ t \in [x_{k}, x_{k+1}], \quad k = 0,1,2,\dots,K.
	$$
	In \cite[Lemma~6]{Matomaki:2012}, it was noted that the function $\sigma(u)$ is the unique solution of the integral equation
	\begin{equation}\label{eq:integral_sigma}
		u\, \sigma(u) = \int_{0}^{u} \sigma(t) \, \alpha(u-t)\, dt,
	\end{equation}
	with the initial condition
	$$
	\sigma(u) = u^{\alpha_{0}-1}, \quad u \in (0, x_{1}].
	$$
	Moreover, $\sigma(u)$ also satisfies the following difference-differential equation:
	\begin{equation}\label{eq:diff_diff_sigma}
		\frac{d}{du} \Bigl( u^{1-\alpha_0} \, \sigma(u) \Bigr) 
		= - \frac{1}{u^{\alpha_0}} \sum_{k=1}^{K'} \sigma(u-x_k) \, (\alpha_{k-1}-\alpha_k).
	\end{equation}
	This difference-differential equation is particularly useful for computing $\sigma(u)$ numerically using mathematical software for our purposes.
\end{rmk}

\section{Proof of Results}

\subsection{The Smoothing Method}
Let $1 \leq Y < \tfrac{X}{2}$. We introduce a smooth compactly supported function $w(x)$ defined by 
$$
w(x) = 
\begin{cases}
	1 & \text{if } x \in [2Y, X],  \\[6pt]
	0 & \text{if } x < Y \text{ or } x > X+Y,
\end{cases}
$$
and satisfying $w^{(r)}(x) \ll_{r} Y^{-r}$ for all $r \geq 0$, where $w^{(r)}(x)$ denotes the $r$-th derivative.

We follow the method of \cite{Lalit-M} (see also \cite{YJGL}). Suppose $f(n) \ll n^{\epsilon}$ for any arbitrarily small $\epsilon > 0$. Then, from \cite[Section~4.1]{Lalit-M},
\begin{equation}\label{Mainest}
	\sum_{n \le X} f(n) 
	= \underset{s=1}{\operatorname{Res}} \left( \frac{X^{s}}{s}\sum_{n\ge 1}\frac{f(n)}{n^{s}} \right) 
	+ V + O(X^{-A'}) + O(Y^{1+\epsilon}),
\end{equation}
where 
\begin{equation}\label{CRT1/2}
	V = \frac{1}{2 \pi i} \int_{\sigma_{0}-iT}^{\sigma_{0}+iT} \tilde w(s) \Bigg(\sum_{n\ge 1} \frac{f(n)}{n^{s}}\Bigg) ds,
\end{equation} 
for any fixed $\sigma_{0} \in (1/2, 1)$. Here $\tilde w(s)$ is the Mellin transform of $w(t)$ and $T = \tfrac{X^{1+\epsilon}}{Y}$.

The Mellin transform is
$$
\tilde w(s) = \int_{0}^{\infty} w(x) x^{s-1}\, dx,
$$
and satisfies
\begin{equation}\label{FourierW}
	\tilde w(s) 
	=  \frac{1}{s(s+1)\cdots(s+m-1)}\int_{0}^{\infty} w^{(m)}(x) x^{s+m-1} dx 
	\ll \frac{Y}{X^{1-\sigma}} \left(\frac{X}{|s|Y}\right)^{m},
\end{equation}
for any $m \geq 1$, where $\sigma = \Re(s)$.

\subsection{Proof of \thmref{thm1.1}}
From Deligne’s bound (resp. Weil’s bound),
$$
\lambda_{f\otimes f \otimes \cdots \otimes_{\ell} f}(n) \ll n^{\epsilon}, 
\qquad 
r_{Q}(n) \ll n^{\epsilon}.
$$
Applying \eqref{Mainest} with $f(n) = \lambda_{f\otimes f \otimes \cdots \otimes_{\ell} f}(n)\, r_{Q}(n)$, we obtain
\begin{equation}\label{EstFN}
	S_{\ell}(f, Q; X )= \sideset{}{^{\flat }} \sum_{n \le X}\lambda_{f\otimes f \otimes \cdots \otimes_{\ell} f}(n)  r_{Q}(n) = \underset{s=1}{\rm Res} \left( \frac{X^{s}}{s} L_{\ell}(f,Q; s) \right) + V_{\ell} + O(X^{-A'}) + O(Y^{1+\epsilon})
\end{equation}
where
\begin{equation}\label{VertEst}
	V_{\ell} = \frac{1}{2 \pi i} \int_{\sigma_{0}-iT}^{\sigma_{0}+iT} \tilde w(s)\, L_{\ell}(f,Q; s)\, ds,
\end{equation}
for $\sigma_{0} \in (1/2, 1)$.

Substituting the decomposition $L_{\ell}(f,Q; s) = L_{\ell}(s)\, U_{\ell}(s)$ (from \lemref{LDecomp}) into \eqref{VertEst}, and using the absolute convergence of $U_{\ell}(s)$ for $\Re(s) > 1/2$, together with \eqref{FourierW}, we deduce
\begin{align*}
	|V_{\ell}| 
	&\ll X^{\sigma_{0}} \int_{-T}^{T} \frac{|L_{\ell}(\sigma_{0}+it)|}{|\sigma_{0}+it|}\, dt \\
	&\ll 2X^{\sigma_{0}} \int_{0}^{T} \frac{|L_{\ell}(\sigma_{0}+it)|}{|\tfrac{1}{2}+\epsilon+it|}\, dt \\
	&\ll X^{\sigma_{0}} \left( \int_{0}^{1} \frac{|L_{\ell}(\sigma_{0}+it)|}{|\sigma_{0}+it|}\, dt 
	+ \int_{1}^{T} \frac{|L_{\ell}(\sigma_{0}+it)|}{|\sigma_{0}+it|}\, dt \right).
\end{align*}
Using convexity bounds for the first integral and a dyadic partition for the second, we get
\begin{equation}\label{VLEstV}
	|V_{\ell}| \ll X^{\sigma_{0}} + X^{\sigma_{0}} \log T \max_{2 \le T_{1} \le T} I_{\ell}(T_{1}),
\end{equation}
where
\begin{equation}\label{VLEst}
	I_{\ell}(T) = \frac{1}{T} \int_{T/2}^{T} L_{\ell}(\sigma_{0}+it)\, dt.
\end{equation}
Taking $\sigma_{0} = \tfrac{1}{2}+\epsilon$ and substituting the decomposition of $L_{\ell}(s)$ (for odd $\ell$) from \eqref{LOdd} into \eqref{VLEst}, we apply Cauchy–Schwarz to obtain
\begin{equation*}
	\begin{split}
		|I_{\ell}(T)| & = \frac{1}{T}  \int_{T/2}^{T}  L_{\ell} \left( \sigma_{0} +it \right) dt  \\
		& \quad \ll 
		\begin{cases}
			\frac{1}{T}   \underset{ \frac{T}{2}  \le t \le  T }{\rm sup}  \left( \displaystyle{\prod_{n=0}^{[\ell/2]-1}} |L(\sigma_{0} +it, \operatorname{sym}^{\ell-2n}f)|^{\left({\ell \choose n}- {\ell \choose {n-1}}\right)}  \right)   \\
			\times   \underset{ \frac{T}{2}  \le t \le  T }{\rm sup}  \left( \displaystyle{\prod_{n=0}^{[\ell/2]-1}} |L(\sigma_{0} +it, \operatorname{sym}^{\ell-2n}f \times \chi_{D})|^{\left({\ell \choose n}- {\ell \choose {n-1}}\right)}  \right)  \\
			\times  \left( \int_{\frac{T}{2}}^{T}   \left| {L(\sigma_{0} +it, f)}  \right|^{2\left({\ell \choose [\ell/2]}- {\ell \choose {[\ell/2]-1}}\right)}  dt  \right)^{\frac{1}{2}} \\
			\times \left( \int_{\frac{T}{2}}^{T}   \left| {L(\sigma_{0} +it, f\times \chi_{D})}  \right|^{2\left({\ell \choose [\ell/2]}- {\ell \choose {[\ell/2]-1}}\right)}  dt  \right)^{\frac{1}{2}} \\
		\end{cases}
	\end{split}
\end{equation*}
\begin{equation*}
	\begin{split}
		 \qquad \qquad \qquad\ll 
		\begin{cases}
			\frac{1}{T}   \underset{ \frac{T}{2}  \le t \le  T }{\rm sup}  \left( \displaystyle{\prod_{n=0}^{[\ell/2]-1}} |L(\sigma_{0} +it, \operatorname{sym}^{\ell-2n}f)|^{\left({\ell \choose n}- {\ell \choose {n-1}}\right)}  \right)   \\
			\times   \underset{ \frac{T}{2}  \le t \le  T }{\rm sup}  \left( \displaystyle{\prod_{n=0}^{[\ell/2]-1}} |L(\sigma_{0} +it, \operatorname{sym}^{\ell-2n}f \times \chi_{D})|^{\left({\ell \choose n}- {\ell \choose {n-1}}\right)}  \right)   \\
			\times   \underset{ \frac{T}{2}  \le t \le  T }{\rm sup}  \left( |L(\sigma_{0} +it, f) L(\sigma_{0} +it, f \times \chi_{D})|^{\left({\ell \choose [\ell/2]}- {\ell \choose {[\ell/2]-1}}-1\right)}  \right)  \\
			\times \left( \int_{\frac{T}{2}}^{T}   \left| {L(\sigma_{0} +it, f)}  \right|^{2} dt  \right)^{\frac{1}{2}}  \left( \int_{\frac{T}{2}}^{T}   \left| {L(\sigma_{0} +it, f \times \chi_{D})}  \right|^{2} dt  \right)^{\frac{1}{2}}. \\
		\end{cases}
	\end{split}
\end{equation*}
Proceeding with hybrid convexity/subconvexity bounds and known second-moment estimates, we arrive at
\begin{equation*}
	\begin{split}
		|I_{\ell}(T)| 
		&  \ll 
		T^{-1} N^{\frac{1}{2} \left(  \displaystyle{\sum_{n=0}^{[\ell/2]}} (\ell-2n) {\left({\ell \choose n}- {\ell \choose {n-1}}\right)} \right)+\epsilon}
		(k T)^{\frac{1}{2} \left(  \displaystyle{\sum_{n=0}^{[\ell/2]}} (\ell-2n+1) {\left({\ell \choose n}- {\ell \choose {n-1}}\right)}  \right)+\epsilon} \\
		& \qquad \times |D|^{\frac{1}{4} \left(  \displaystyle{\sum_{n=0}^{[\ell/2]}} (\ell-2n+1) {\left({\ell \choose n}- {\ell \choose {n-1}}\right)}  \right)+\epsilon}
	\end{split}.
\end{equation*}
Let us consider the following  identity: $ {\ell \choose {n}}- {\ell \choose {n-1}} = \frac{\ell-2n+1}{\ell-n+1} {\ell \choose {n}}$ when $n >0$ and $1$ when $n=0$ (The proof of above identity follows clearly from the definition). This enables us to get
\begin{equation*}
	\begin{split}
		|I_{\ell}(T)| 
		&  \ll 
		T^{-1}  {N}^{\frac{1}{2} \displaystyle{ \left[\sum_{n=0}^{[\ell/2]} {\frac{(\ell-2n+1)(\ell-2n)}{\ell-n+1} {\ell \choose {n}}} \right]} +\epsilon }  (kT)^{\frac{1}{2} \displaystyle{ \left[\sum_{n=0}^{[\ell/2]} {\frac{(\ell-2n+1)^{2}}{\ell-n+1} {\ell \choose {n}}} \right]} +\epsilon} \\
		& \qquad \times |D|^{\frac{1}{4} \displaystyle{ \left[\sum_{n=0}^{[\ell/2]-1} {\frac{(\ell-2n+1)^{2}}{\ell-n+1} {\ell \choose {n}}} \right]} +\epsilon},
	\end{split}
\end{equation*}
We substitute the value of  $|I_{\ell}(T)|$ to get
\begin{equation}\label{VEst}
	\begin{split}
		|V_{\ell}|   & \ll  X^{\frac{1}{2}+\epsilon}
		T^{-1} N^{\frac{1}{2} A + \epsilon} (kT)^{\frac{1}{2} B + \epsilon} |D|^{\frac{1}{4} B + \epsilon}\\
	\end{split},
\end{equation}
where $A=\sum_{n=0}^{[\ell/2]} {\frac{(\ell-2n+1)(\ell-2n)}{\ell-n+1} {\ell \choose {n}}} $ and $B= \sum_{n=0}^{[\ell/2]} {\frac{(\ell-2n+1)^{2}}{\ell-n+1} {\ell \choose {n}}}$.
Thus, substituting  the estimate of $V_{\ell}$ from \eqref{VEst} in \eqref{EstFN},  we have (for odd $\ell$)
\begin{equation*}
	S_{\ell}(f,Q; X )= \sideset{}{^{\flat }} \sum_{n \le X}\lambda_{f\otimes  \cdots \otimes_{\ell} f}(n)r_Q(n) = O\left(  X^{\frac{1}{2}+\epsilon}
	T^{-1} N^{\frac{1}{2} A + \epsilon} (kT)^{\frac{1}{2} B + \epsilon} |D|^{\frac{1}{4} B + \epsilon} \right)   + O(Y^{1+\epsilon}) + O(X^{-A'}).
\end{equation*}
Finally, with $T = \tfrac{X^{1+\epsilon}}{Y}$ and 
$$
Y = X^{1-\tfrac{1}{B}+\epsilon} \left( N^{A} (k|D|^{1/2})^{B} \right)^{\tfrac{1}{B}+\epsilon},
$$
we conclude
$$
S_{\ell}(f, Q; X) 
= O\!\left( X^{1-\tfrac{1}{B}+\epsilon} \left( N^{A} (k|D|^{1/2})^{B} \right)^{\tfrac{1}{B}+\epsilon} \right).
$$
This completes the proof.  

\subsection*{A lower bound for the sum $S_{\ell}(f,Q; Y^{u_0})$}
Let $u_{0}=u_{0}(\ell)$.  Recall the multiplicative function $h_{Y}$, supported on square-free positive integers, is defined at primes by
$$
h_{Y}(p)=
\begin{cases}
	\alpha\!\left(\tfrac{\log p}{\log Y}\right), & \text{if } p \le Y \text{ and } p \nmid N,\\[4pt]
	-2, & \text{if } p > Y \text{ and } p \nmid N,\\[4pt]
	0, & \text{if } p \mid N,
\end{cases}
$$
where $\alpha:[0,1]\to[-2,2]$ satisfies $\alpha(0)=2$ and, for $m\in\mathbb{N}$,
$$
\alpha(t)=2\cos\!\left(\tfrac{\pi}{m+1}\right)\quad\text{whenever}\quad \tfrac{1}{m+1}<t\le\tfrac{1}{m}.
$$
Let $\ell$ be an odd positive integer.  For a prime $p\nmid N$ the Euler coefficient of the $\ell$-fold product $L$-function equals $\lambda_f(p)^\ell$, i.e.
$$
\lambda_{f\otimes\cdots\otimes_{\ell} f}(p)=\lambda_f(p)^\ell.
$$
By multiplicativity this gives $\lambda_{f\otimes\cdots\otimes_{\ell} f}(n)=\lambda_f(n)^\ell$ for all $n$ coprime to $N$.

Assume that $\lambda_{f\otimes\cdots\otimes_{\ell} f}(n)\ge0$ for every square-free $n\le Y$.  Then $\lambda_f(n)\ge0$ for each such $n$, and the Hecke relations together with the assumption $\lambda_f(p)\ge0$ for every prime $p\le Y^{1/m}$ imply
$$
\lambda_f(p)\ge 2\cos\!\left(\tfrac{\pi}{m+1}\right).
$$
Hence for all primes $p\le Y$ we have
$$
\lambda_{f\otimes\cdots\otimes_{\ell} f}(p)\ge h_Y(p)^\ell,
$$
and multiplicativity extends this to all $n\le Y$ (see \cite[Section~2]{KowalskiLauSoundWu:2010}).  Since $r_Q^*(n)\ge0$, it follows that
$$
\lambda_{f\otimes\cdots\otimes_{\ell} f}(n)\,r_Q^*(n)\ge h_Y(n)^\ell\,r_Q^*(n)
\qquad\text{for every square-free }n\le Y.
$$
Arguing as in \cite{KowalskiLauSoundWu:2010}, for each $u\le u_{0,\ell}$ (take $X=Y^u$) we obtain the lower bound
\begin{equation}\label{LBound}
	S_{\ell}(f,Q;Y^{u})
	= \sideset{}{^{\flat }}\sum_{\substack{n\le Y^{u}\\ \gcd(n,N)=1}} \lambda_{f\otimes\cdots\otimes_{\ell} f}(n)\,r_Q^*(n)
	\ge \sideset{}{^{\flat }}\sum_{\substack{n\le Y^{u}\\ \gcd(n,N)=1}} h_Y(n)^{\ell}\,r_Q^*(n) >0.
\end{equation}
For completeness we sketch the convolution argument that gives \eqref{LBound}.  Define multiplicative functions
$$
\lambda_{f\otimes\cdots\otimes_{\ell} f}^*(n):=\lambda_{f\otimes\cdots\otimes_{\ell} f}(n)\,r_Q^*(n),\qquad
h_Y^{\ell,*}(n):=h_Y(n)^{\ell}\,r_Q^*(n).
$$
There exists a multiplicative $g_Q$ with
$$
\lambda_{f\otimes\cdots\otimes_{\ell} f}^*=g_Q * h_Y^{\ell,*}.
$$
At primes this gives
$$
g_Q(p)=\lambda_{f\otimes\cdots\otimes_{\ell} f}(p)\,r_Q^*(p)-h_Y(p)^{\ell}\,r_Q^*(p).
$$
If $p$ is not represented by $Q(\vect{x})$ then $r_Q^*(p)=0$ and hence $g_Q(p)=0$; if $p$ is represented by $Q$ then by the definition of $h_Y(p)$ we have $g_Q(p)\ge0$ for all $p\nmid N$ with $p\le Y$.  Therefore $g_Q(n)\ge0$ for every square-free $n\le Y$ and $g_Q(1)=1$.

Using the Dirichlet convolution and restricting to square-free $n$ coprime to $N$,
\begin{align*}
	S_{\ell}(f,Q;Y^{u})
	&= \sideset{}{^{\flat }}\sum_{\substack{n\le Y^{u}\\ \gcd(n,N)=1}} \lambda_{f\otimes\cdots\otimes_{\ell} f}^*(n)
	= \sideset{}{^{\flat }}\sum_{\substack{n\le Y^{u}\\ \gcd(n,N)=1}} \sum_{d\mid n} g_Q(d)\,h_Y^{\ell,*}\!\left(\tfrac{n}{d}\right)\\
	&= \sideset{}{^{\flat }}\sum_{\substack{d\le Y^{u}\\ \gcd(d,N)=1}} g_Q(d)\;
	\sideset{}{^{\flat }}\sum_{\substack{m\le Y^{u}/d\\ \gcd(m,N)=1}} h_Y(m)^{\ell}\,r_Q^*(m).
\end{align*}
Since $g_Q(d)\ge0$ for all square-free $d\le Y$ and $g_Q(1)=1$, the inner sum when $d=1$ already equals the rightmost sum in \eqref{LBound}, whence the inequality in \eqref{LBound} follows.

Hence there is a positive lower bound for $S_{\ell}(f,Q;Y^{u_0})$, in fact for some $u_{0,\ell}>1$ one has that the main positive contribution is
$$
\sum_{\substack{n\le Y^{u_{0,\ell}}\\ \gcd(n,N)=1}} h_Y(n)^{\ell}\,r_Q^*(n),
$$
as stated in \propref{PropLower}.

\begin{rmk}
	For the step function $\alpha$ one shows that $\sigma(u_{0,\ell})>0$.  Consequently, for $u=u_{0,\ell}$,
	\begin{equation}\label{PBound}
		\sum_{\substack{n\le Y^{u_{0,\ell}}\\ \gcd(n,N)=1}} h_Y(n)^{\ell}\,r_Q^*(n)
		= \Bigl(\sigma(u_{0,\ell})+o_{\alpha,\ell,U}(1)\Bigr)
		\frac{P(1)L(1,\chi_D)^{\alpha_0^\ell}}{\Gamma(\alpha_0^\ell)}
		(\log Y)^{\alpha_0^\ell-1} Y^{u_{0,\ell}} >0.
	\end{equation}
\end{rmk}

\subsection{Proof of \thmref{thm1.2}}
By definition, let $n_{f\otimes\cdots\otimes_{\ell} f,D}$ be the largest integer $n\in\mathbb{N}$ such that
$$
\lambda_{f\otimes\cdots\otimes_{\ell} f}(n)\ge0,\qquad \gcd(n,N)=1,
$$
and $n=Q(\vect{x})$ for some $\vect{x}\in\mathbb{Z}^2$.  Write $n_{f\otimes\cdots\otimes_{\ell} f,D}=Y$.  In other words, for every $n\le Y$ with $\gcd(n,N)=1$ that is represented by $Q$ we have $\lambda_{f\otimes\cdots\otimes_{\ell} f}(n)\ge0$.

To estimate $Y$ compare the lower bound \eqref{PBound} (via \eqref{LBound}) with the upper bound obtained from \thmref{thm1.1}.  For $u>1$ define
$$
S_{\ell}(f,Q;Y^u)
= \sideset{}{^{\flat }}\sum_{\substack{n=Q(\vect{x})\le Y^u\\ \vect{x}\in\mathbb{Z}^2,\;\gcd(n,N)=1}} \lambda_{f\otimes\cdots\otimes_{\ell} f}(n)
= \sideset{}{^{\flat }}\sum_{\substack{n\le Y^u\\ \gcd(n,N)=1}} \lambda_{f\otimes\cdots\otimes_{\ell} f}(n)\,r_Q^*(n).
$$
From \thmref{thm1.1} and \eqref{PBound} (valid for $u\le u_{0,\ell}$) we have
$$
0<\frac{P(1)L(1,\chi_D)^{2^{\ell}}}{\Gamma(\alpha_0^\ell)}\,Y^{u}\log(Y^{u})
\ll \sideset{}{^{\flat }}\sum_{\substack{n\le Y^{u}\\ \gcd(n,N)=1}} h_Y(n)^{\ell}\,r_Q^*(n)
\le S_{\ell}(f,Q;Y^{u})
$$
and by the upper bound of \thmref{thm1.1} (see the end of the proof of that theorem),
$$
S_{\ell}(f,Q;Y^{u}) \ll Y^{u\bigl(1-\tfrac{2}{B}+\epsilon\bigr)}
\Bigl(N^{A}~k^{B}~|D|^{B/2}\Bigr)^{\tfrac{1}{B}+\epsilon}.
$$
Combining these inequalities yields
$$
Y^{u}\log(Y^{u})L(1,\chi_D)^{2^{\ell}} \ll Y^{u\bigl(1-\tfrac{2}{B}+\epsilon\bigr)}
\Bigl(N^{A}~k^{B}~|D|^{B/2}\Bigr)^{\tfrac{1}{B}+\epsilon}.
$$
Therefore,  
$$
Y^{2u/B+\epsilon} \ll \Bigl(N^{A}~k^{B}~|D|^{B/2}\Bigr)^{\tfrac{1}{B}+\epsilon}L(1,\chi_D)^{-2^{\ell}}
\quad\Longrightarrow\quad
Y \ll \Bigl(N^{A}~k^{B}~|D|^{B/2}\Bigr)^{\tfrac{1}{2u}+\epsilon}L(1,\chi_D)^{\frac{-2^{\ell-1}B}{u}}.
$$
Substituting $L(1,\chi_D)=\frac{2\pi h(D)}{w_D\sqrt{|D|}} $ and taking $u=u_{0,\ell}$ completes the proof of \thmref{thm1.2}.

\bigskip
\noindent
\textbf{Acknowledgement:}   
The authors express their gratitude to the Stat-Math Unit, ISI Delhi, for providing the opportunity to pursue research activities. The third author gratefully acknowledges the support of DST-INSPIRE, DST, Government of India, through the INSPIRE Faculty Fellowship, which has greatly contributed to academic growth and research opportunities.

\end{document}